\documentclass[12pt]{article}
\title{Interpretable sets in dense o-minimal structures}
\author{Will Johnson}

\usepackage{amsmath, amssymb, amsthm}    	
\usepackage{fullpage} 	
\usepackage{amscd}
\usepackage{hyperref}
\usepackage[all]{xy}
\usepackage[T1]{fontenc}
\usepackage{centernot}
\usepackage{titlefoot}

\DeclareMathOperator*{\forkindeb}{\raise0.2ex\hbox{\ooalign{\hidewidth$\vert$\hidewidth\cr\raise-0.9ex\hbox{$\smile$}}}}
\newcommand{\forkindep}{\forkindeb^\textrm{\th}}

\newcommand{\Aut}{\operatorname{Aut}}

\newcommand{\dcl}{\operatorname{dcl}^{\text{eq}}}
\newcommand{\tp}{\operatorname{tp}}

\newcommand{\bd}{\operatorname{bd}}
\newcommand{\ter}{\operatorname{int}}

\newtheorem{theorem}{Theorem}[section] 
\newtheorem{lemma}[theorem]{Lemma}
\newtheorem{claim}[theorem]{Claim}
\newtheorem{definition}[theorem]{Definition}

\newtheorem{question}[theorem]{Question}
\newtheorem{remark}[theorem]{Remark}
\newtheorem{example}[theorem]{Example}
\newtheorem{proposition}[theorem]{Proposition}

\newcommand{\Qq}{\mathbb{Q}}

\newcommand{\Rr}{\mathbb{R}}
\newcommand{\Rp}{\mathbb{RP}}
\newcommand{\Zz}{\mathbb{Z}}
\newcommand{\Nn}{\mathbb{N}}
\newcommand{\Cc}{\mathbb{C}}

\begin{document}
\maketitle\unmarkedfntext{


The author would like to thank Kobi Peterzil for asking the question that prompted
this paper, as well as Tom Scanlon for reviewing earlier versions
of this paper that were part of the author's PhD dissertation.

Parts of this material are based upon work supported by the
  National Science Foundation Graduate Research Fellowship under Grant
  No. DGE 1106400.  Any opinion, findings, and conclusions or
  recommendations expressed in this material are those of the author
  and do not necessarily reflect the views of the National Science
  Foundation.
}

\begin{abstract}
We give an example of a dense o-minimal structure in which there is a definable quotient that cannot be eliminated, even after naming parameters.  Equivalently, there is an interpretable set which cannot be put in parametrically definable bijection with any definable set.  This gives a negative answer to a question of Eleftheriou, Peterzil, and Ramakrishnan.  Additionally, we show that interpretable sets in dense o-minimal structures admit definable topologies which are ``tame'' in several ways: (a) they are Hausdorff, (b) every point has a neighborhood which is definably homeomorphic to a definable set, (c) definable functions are piecewise continuous, (d) definable subsets have finitely many definably connected components, and (e) the frontier of a definable subset has lower dimension than the subset itself.
\end{abstract}

\section{Introduction}
Let us say that a structure $M$ has \emph{parametric elimination of
  imaginaries} if given any $M$-definable set $X$ and $M$-definable
equivalence relation $E$ on $X$, there is an $M$-definable map
eliminating the quotient $X/E$.  Replacing ``$M$-definable'' with
``0-definable'' gives the usual notion of elimination of imaginaries,
which is a stronger condition.

It is well-known that o-minimal expansions of ordered abelian groups have parametric elimination of imaginaries.  When working with o-minimal structures, it is common to assume that the structure expands an ordered abelian group, or even an ordered field.  This assumption simplifies life, and holds in most o-minimal structures arising in applications of o-minimality.  Nevertheless, some o-minimal structures do not expand ordered abelian groups, and one can pose the following question:
\begin{question} \label{q1}
  Do all o-minimal structures have parametric elimination of
  imaginaries?
\end{question}
This question was first asked by Eleftheriou, Peterzil, and Ramakrishnan in \cite{interpretable-groups}.
They gave a partial answer, proving that an o-minimal quotient $X/E$
can be eliminated whenever it admits a definable group structure, as
well as when $\dim(X/E) =1$.

We answer Question~\ref{q1} in the negative in \S\ref{sec:bad-quot}.
Specifically, we give an o-minimal expansion of $(\Rr,\le)$ in which
there is a 0-definable quotient $X/E$ which cannot be eliminated over
any set of parameters.

A structure $M$ has parametric elimination of imaginaries if every
interpretable set in $M$ can be put in definable bijection with a
definable set.  The negative answer to Question~\ref{q1} therefore
means that o-minimal structures can have exotic interpretable sets
which are intrinsically different from definable sets.

O-minimality provides many tools for working with definable sets, and
it is natural to wonder which of these tools can be generalized to
interpretable sets.  For example, Peterzil and Kamenkovich
 generalized the dimension and Euler
characteristic machinery to interpretable sets
 in \cite{interpdim} and \cite{kamenkovich}, respectively.

As a step in this direction, we show in \S\ref{sec:positive} that
interpretable sets $X/E$ in dense o-minimal theories can be given nice
definable topologies.  More precisely, we show in
Theorem~\ref{main-theorem} that the quotient topology on $X/E$ is a
Hausdorff definable topology, \emph{provided} one first discards a set
of low dimension from $X$.\footnote{Without this proviso, one can produce
pathological examples such as the line with doubled origin.  Indeed, if
$X = \Rr \times \{0,1\}$ and $E$ is the equivalence relation generated by
\[ (x,0)E(x,1) \text{ for } x \ne 0,\]
then the quotient $X/E$ is the line with doubled origin.}

Using this theorem, we show that interpretable sets admit Hausdorff definable
topologies satisfying certain ``tameness'' properties, including the
following:
\begin{itemize}
\item Every definable subset has finitely many definably connected
  components.
\item Every definable map is continuous off a set of low dimension.
\end{itemize}
For a precise statement, see Theorem~\ref{addendum}, which
is proven in \S\ref{sec:admiss}.

\subsection{Notation and conventions}
``Definable'' will mean ``definable with parameters,'' and
``$A$-definable'' will mean ``definable with parameters from $A$''.
We will write ``0-definable'' as shorthand for
``$\emptyset$-definable.''

When talking about sets, a ``definable set'' means a definable subset
of a power of the home sort, and an ``interpretable set'' means a
definable set in $T^{\text{eq}}$.  Outside of this distinction, we
will always say ``definable'' instead of ``interpretable.''  For
example, we will talk about definable subsets of interpretable sets,
and definable maps between interpretable sets, rather than
``interpretable subsets'' or ``interpretable maps''.  We will say that a subset of an interpretable set is ``ind-definable'' (over some parameters $A$) if it is a union of $A$-definable subsets.

A ``definable quotient'' is a pair $X/E$ consisting of a definable set
$X$ and a definable subset $E \subseteq X \times X$ defining an
equivalence relation on $X$.  The quotient can be ``eliminated'' if
one of the following equivalent conditions is true:
\begin{itemize}
\item There is a definable bijection between the interpretable set
  $X/E$ and some definable set $Y \subseteq M^k$.
\item There is some definable map $f : X \to M^k$ such that
  \begin{equation*}
    xEx' \iff f(x) = f(x') \qquad \forall x,x' \in X
  \end{equation*}
\end{itemize}

``O-minimal'' will mean dense o-minimal, i.e., we require o-minimal
structures to expand dense linear orders without endpoints.

In an o-minimal structure, $\dim(X)$ will denote the standard
o-minimal dimension of a definable or interpretable set $X$ (see
\cite{interpdim} for the interpretable case).  The o-minimal rank of a
finite tuple $a$ over a set of parameters $S$ will be denoted
$\dim(a/S)$; this is the minimum of $\dim(X)$ for $S$-definable $X \ni
a$.  We will write $\forkindep$ to denote thorn-forking independence,
so $a \forkindep_C B$ means $\dim(a/BC) = \dim(a/C)$.

In a topological space, the interior, boundary, frontier, and closure
of a set $X$ will be denoted $\ter(X)$, $\bd(X)$,
$\partial X$, and $\overline{X}$.  Thus
\begin{align*}
  \bd(X) &= \overline{X} \setminus \ter(X) \\
  \partial X &= \overline{X} \setminus X
\end{align*}
An ``embedding'' will be a continuous map that is a homeomorphism onto its image.

If $E$ is an equivalence relation on a set $X$, and $X' \subseteq X$,
we will write $X'/E$ to indicate $X'/(E \restriction X')$.

A map $f : P_1 \to P_2$ between two posets will be called \emph{order-preserving} if
\[ x \le y \implies f(x) \le f(y)\]
and \emph{order-reversing} if
\[ x \le y \implies f(x) \ge f(y).\]
(Usually the posets will be powersets with inclusion ordering.)

If $X$ is a definable set in a structure $M$, then $\ulcorner X \urcorner$ will denote a canonical parameter for $X$, i.e., a finite tuple from $M^{eq}$ fixed pointwise by exactly the automorphisms that fix $X$ setwise.  If $r$ is a real number, $\lceil r \rceil$ and $\lfloor r \rfloor$ will denote the ceiling and floor of $r$, respectively.

If $X$ is a definable or interpretable set in a structure $M$, a
topology on $X$ is ``definable'' if there is a definable family of
subsets of $X$ forming a basis of opens.  This means that there is a definable
relation $U \subseteq X \times M^k$ for which the sets
\begin{equation*}
  U_{\vec{a}} := \{x \in X | (x,\vec{a}) \in U\} \text{ for } \vec{a} \in M^k
\end{equation*}
form a basis for the topology.  A ``definable topological space'' is
an interpretable set together with a definable topology.


If $X$ is a definable topological space in an o-minimal structure
$(M,\le,\ldots)$, we will say that $X$ is \emph{Euclidean} at a point
$x \in X$ if there is a definable homeomorphism between an open
neighborhood of $x$ in $X$ and an open subset of $M^k$ for some $k$.
We will say that $X$ is ``locally Euclidean'' if $X$ is Euclidean at
every $x \in X$.  (Note that $k$ might depend on $x$.)

\subsection{Statement of results}

\begin{proposition}
  There is a (dense) o-minimal structure $M$ containing an
  interpretable set which cannot be put in $M$-definable bijection
  with any $M$-definable set.
\end{proposition}

\begin{theorem} \label{main-theorem}
  Fix an o-minimal structure $M$.  Let $X \subseteq M^k$ be a
  definable set and $E$ be a definable equivalence relation on $X$.
  Then we can write $X$ as a disjoint union $X' \cup X_0$ satisfying
  the following conditions:
  \begin{enumerate}
  \item \label{cond1} $X'$ is open in $X$
  \item \label{cond2} $\dim(X_0) < \dim(X)$ or $X_0 = \emptyset$.
  \item \label{cond3} The quotient topology on $X'/E$ is definable,
    Hausdorff, and locally Euclidean
  \item \label{cond4} If $X''$ is any open subset of $X'$, the map of
    quotient spaces
    \begin{equation*}
      X''/E \hookrightarrow X'/E
    \end{equation*}
    is continuous, and in fact an open embedding.
  \end{enumerate}
\end{theorem}
Condition~\ref{cond2} means that $X'$ is ``generic'' in $X$ in a
certain sense.  Condition~\ref{cond4} shows that the quotient topology
is somewhat independent of the choice of $X'$: as long as we have
chosen a sufficiently small generic open subset of $X$, the quotient
topology will agree.

\begin{definition}
  A Hausdorff topology on an interpretable set $Y$ is
\emph{admissible} if there is a definable surjection $f : X
\twoheadrightarrow Y$ where $X$ is a definable subset of $M^n$, such that $f$ is a continuous open map with respect to the standard topology on $X$.
\end{definition}
The
next result says that admissible locally Euclidean topologies exist,
and share many properties with the standard topology on
$M^k$.
\begin{theorem} \label{addendum}
  Fix an o-minimal structure $M$.
  \begin{enumerate}
  \item Every interpretable set can be endowed with an admissible locally Euclidean
    topology.
  \item Admissible topologies are definable.
  \item If $Y$ is an admissible locally Euclidean topological space and $D$ is
    a non-empty definable subset of $Y$, then
    \begin{enumerate}
    \item $D$ has finitely many definably connected components.
    \item $\dim \partial D < \dim D$.
    \item There is a point $p \in D$ such that $\dim N \cap D = \dim
      D$ for every neighborhood $N$ of $p$.  In other words, the local
      dimension of $D$ at $p$ equals the global dimension of $D$.
    \end{enumerate}
  \item If $f : Y \to Y'$ is a definable map between two admissible
    locally Euclidean topological spaces, then $f$ is continuous on a
    dense open subset of $Y$.  Moreover, $Y$ can be written as a
    finite disjoint union of locally closed definable subsets, on
    which the restriction of $f$ is continuous.
  \end{enumerate}
\end{theorem}
Note that there are other ways to put locally Euclidean definable
topologies on interpretable sets, such as the discrete topology.
However, the discrete topology fails to satisfy many of the conditions
listed above, such as 3a, 3c, and 4.

\section{A pathological quotient} \label{sec:bad-quot}
In this section, we give an example of an o-minimal structure in which
parametric elimination of imaginaries fails, namely
\begin{equation*}
  (\Rr,\le,R)
\end{equation*}
where $R(x_0,\ldots,x_5)$ is the 5-ary predicate holding if and only if
\begin{equation*}
  \frac{\cos(x_1 - x_0)}{\sin(x_1 - x_0)} - \frac{\cos(x_2 -
    x_0)}{\sin(x_2 - x_0)} = \frac{\cos(x_3 - x_0)}{\sin(x_3 - x_0)} -
  \frac{\cos(x_4 - x_0)}{\sin(x_4 - x_0)} \text{ and } \bigwedge_{i =
    1}^4 x_0 < x_i < x_0 + \pi.
\end{equation*}
\subsection{A toy example} \label{sec:toy}
We first discuss the simplest example of an o-minimal theory which
lacks elimination of imaginaries.  Let $M = (\Rr,\le,E)$, where
$E(\vec{x})$ is the 4-ary relation
\begin{equation*}
  E(x_1,\ldots,x_4) \iff x_1 - x_2 = x_3 - x_4.
\end{equation*}
The relation $E$ defines an equivalence relation on the set $X :=
\Rr^2$.  The quotient $X/E$ cannot be 0-definably eliminated, and this
can be seen using automorphisms.  Suppose for the sake of
contradiction that there is a 0-definable injection $X/E
\hookrightarrow M^k$ for some $k$.  Consider the automorphisms
\begin{align*}
  \sigma_1(x) &:= x + 1 \\
  \sigma_2(x) &:= 2x.
\end{align*}
of the structure $M$.  Let $e \in X/E \subseteq M^{\text{eq}}$ be the
$E$-equivalence class of $(0,1) \in X$.  Then one verifies easily that
\begin{align*}
  \sigma_1(e) & = e \\
  \sigma_2(e) & \ne e
\end{align*}
Let $\vec{r}$ denote $f(e)$.  Then $e$ and $\vec{r}$ are
inter-definable over $\emptyset$, so
\begin{align*}
  \sigma_1(\vec{r}) & = \vec{r} \\
  \sigma_2(\vec{r}) & \ne \vec{r}
\end{align*}
However, $\vec{r}$ is a tuple of elements from $M$.  By inspection,
every element of $M$ fixed by $\sigma_1$ is fixed by $\sigma_2$,
yielding a contradiction.\footnote{In this toy example, $\sigma_1$ has no fixed points in $M$, and so $\sigma_2$'s only role is to rule out the possibility that $\vec{r}$ is the tuple of length 0.  Later, we will use the same argument in a more complicated situation where $\sigma_1$ has fixed points.}

The structure $M$ gives an example of an o-minimal theory which does not have
elimination of imaginaries.  Nevertheless, after naming two constants,
this example has a strong form of elimination of imaginaries: every non-empty definable set $X$ contains an $\ulcorner X \urcorner$-definable point.  So this is not yet an example of an o-minimal
structure in which parametric elimination of imaginaries fails.
However, this toy example will play a role in the construction below.

For future reference, we record the configuration that showed that a
quotient was not eliminated:
\begin{lemma} \label{ei-fail-test}
  Let $M$ be a structure, $A$ be a small set of parameters, and $X/E$
  be an $A$-definable quotient.  Suppose there exist $\sigma_i \in
  \Aut(M/A)$ for $i = 1,2$ such that
  \begin{itemize}
  \item Every element of $M$ fixed by $\sigma_1$ is fixed by $\sigma_2$
  \item Some element of $X/E$ fixed by $\sigma_1$ is not fixed by $\sigma_2$.
  \end{itemize}
  Then there is no $A$-definable injection $X/E \hookrightarrow M^k$,
  so the quotient $X/E$ cannot be eliminated over $A$.
\end{lemma}

\subsection{Preliminaries}
Let $\Rp^1 = \Rr \cup \{\infty\}$ be the real projective line.  The
group of linear fractional transformations $x \mapsto \frac{ax + b}{cx
  + d}$ acts transitively on $\Rp^1$, and the stabilizer of $\infty$
is exactly the group of affine transformations $x \mapsto ax + b$.

For $x_0,\ldots,x_5 \in \Rp^1$, let $P(x_0,\ldots,x_4)$ indicate that
\begin{equation*}
  f(x_1) - f(x_2) = f(x_3) - f(x_4)
\end{equation*}
for any/every linear fractional transformation $f$ mapping $x_5$ to
$\infty$.  This is well-defined because $f$ is determined up to an
affine transformation, and affine transformations preserve the 4-ary
relation $y_1 - y_2 = y_3 - y_4$.

\begin{remark}\label{proto-affine}
  Any linear fractional transformation (and in particular, any affine
  transformation) preserves the predicate $P$.
\end{remark}

We will write $\cot \theta$ and $\tan \theta$ for the cotangent and
tangent of the angle $\theta$.
\begin{remark}\label{cot}  
  For fixed $\alpha \in \Rr$, there is a linear fractional
  transformation mapping $\tan x \mapsto \cot( x - \alpha)$, by the
  trigonometric angle-sum formulas.  This transformation sends $\tan
  \alpha \mapsto \cot ( \alpha - \alpha) = \infty$, and so
  \begin{align*}
    & P(\tan \alpha ,\tan x_1 ,\ldots,\tan x_4 ) \\
    & \iff \cot(x_1 - \alpha) - \cot(x_2 - \alpha) = \cot(x_3 - \alpha) - \cot(x_4 - \alpha)
  \end{align*}
\end{remark}

We also record the trivial example
\begin{equation}
  P(\infty, x_1, \ldots, x_4) \iff x_1 - x_2 = x_3 - x_4
  \label{infty-case}
\end{equation}

\subsection{Details of the construction}
Let $M$ be the structure $(\Rr,\le,\iota,\tilde{P})$ where
\begin{itemize}
\item $\iota(x) = x + \pi$
\item $\tilde{P}(x_0,x_1,\ldots,x_4)$ holds if
  \begin{equation*}
    P(\tan(x_0),\ldots,\tan(x_4)) \wedge \bigwedge_{i = 1}^4 x_0 < x_i
    < \iota(x_0).
  \end{equation*}
  By Remark~\ref{cot}, $\tilde{P}(x_0,\ldots,x_4)$ holds if and only if
  $\{x_1,\ldots,x_4\} \subseteq (x_0,x_0+\pi)$ and
  \begin{equation*}
    \cot(x_1 - x_0) - \cot(x_2 - x_0) = \cot(x_3 - x_0) - \cot(x_4 - x_0)
  \end{equation*}
\end{itemize}
Let $N$ be the structure $(\Zz \times \Rp^1, \le, \iota, \tilde{P})$ where
\begin{itemize}
\item $\le$ is the lexicographic ordering on $\Zz \times \Rp^1$, where $\Rp^1$ is ordered by putting $\infty > \Rr$.
\item $\iota$ is the map $(n,x) \mapsto (n+1,x)$
\item $\tilde{P}(x_0,x_1,\ldots,x_4)$ holds if
  \begin{equation*}
    P(\pi(x_0),\ldots,\pi(x_4)) \wedge \bigwedge_{i = 1}^4 x_0 < x_i < \iota(x_0)
  \end{equation*}
  where $\pi : \Zz \times \Rp^1 \to \Rp^1$ is the projection.
\end{itemize}
It is easy to verify that there is an isomorphism $M \stackrel{\sim}{\to} N$ given by
\begin{equation*}
  x \mapsto \left( \left\lceil \frac{x}{\pi} + \frac{1}{2}
  \right\rceil , \tan x \right).
\end{equation*}
The map preserves $\tilde{P}$ essentially because the following diagram commutes
  \begin{equation*}
	  \xymatrix{M \ar[r] \ar[rd]_{\tan} & N \ar[d]^{\pi} \\ & \Rp^1}.
  \end{equation*}

The two structures $M$ and $N$ are o-minimal, because $M$ is a definable reduct of
\begin{equation*}
  (\Rr,\le,+,\cdot,\sin \restriction [0,\pi], \cos \restriction [0,\pi]),
\end{equation*}
which is o-minimal by Gabrielov's theorem (see e.g. Theorem 4.6 in
\cite{denef-vdd}).

For any $a \in M$, let
\begin{equation*}
  X_a = \{(x,y) : a < x < y < \iota(a)\} \subseteq M^2
\end{equation*}
and let $E_a$ be the equivalence relation on $X_a$ given by
\begin{align*}
  (x,y)E_a(x',y') &\iff \tilde{P}(a,x,y,x',y') \\ & \iff \cot(x - a) -
  \cot(y - a) = \cot(x' - a) - \cot(y' - a)
\end{align*}
Via the isomorphism, the same definitions make sense in $N$.
\begin{example}
  \label{example-play-ugly}
  In the structure $N$, consider the case $a = (n-1,\infty)$.  The
  open interval from $a$ to $\iota(a) = (n,\infty)$ consists of points
  $(n,x)$ with $x \in \Rr$.  Abusing notation and identifying $(n,x)$
  with $x$, we have
  \begin{align*}
    X_a &= \{(s,t) \in \Rr^2 : s < t\} \\
    (s,t)E_a(s',t') &\iff s - t = s' - t' \\
  \end{align*}
  So $X_a/E_a$ is the toy example of \S\ref{sec:toy}.
\end{example}

\begin{lemma}\label{first-taus}
  In the structures $M$ and $N$, there are automorphisms $\tau_1,
  \tau_2$ such that
  \begin{enumerate}
  \item Every element of the home sort fixed by $\tau_1$ is fixed by
    $\tau_2$
  \item The set of elements fixed by $\tau_1$ is unbounded above
  \item If $a$ is fixed by $\tau_1$, then under the induced action
    on $M^{\text{eq}}$ or $N^{\text{eq}}$, $\tau_1$ fixes every element of $X_a/E_a$
    and $\tau_2$ fixes no elements of $X_a/E_a$.
  \end{enumerate}
\end{lemma}
\begin{proof}
  By the isomorphism $M \cong N$, we only need to consider the case of
  $N$.  In this case, let
  \begin{align*}
    \tau_1((n,x)) &= (n,x+1)  \\
    \tau_2((n,x)) &= (n,2x)
  \end{align*}
  These maps are indeed automorphisms; $\tilde{P}$ is preserved
  because of Remark~\ref{proto-affine}.  The fixed points of $\tau_1$
  are exactly the points $(n,\infty)$, which are cofinal and fixed by
  $\tau_2$.  For part 3, suppose $a = (n-1,\infty)$.  Under the
  identification of Example~\ref{example-play-ugly},
  \begin{align*}
    X_a &= \{(s,t) \in \Rr^2 : s < t\} \\
    (s,t)E_a(s',t') &\iff s - t = s' - t' \\
    \tau_1(s) &= s + 1 \\
    \tau_2(s) &= 2s
  \end{align*}
  As in \S\ref{sec:toy}, $\tau_1$ fixes $X_a/E_a$ pointwise.  In
  contrast, $\tau_2$ moves every point, because
  \begin{equation*}
    s < t \implies s - t \ne 2s - 2t
  \end{equation*}
\end{proof}

Now let $M^*$ be an $\aleph_1$-saturated ultrapower of $M$; there are
canonical extensions of $\tau_1$ and $\tau_2$ to $M^*$ having the same
first-order properties.  In particular, the properties listed in
Lemma~\ref{first-taus} continue to hold.

The following lemma allows us to glue automorphisms across Dedekind
cuts in $M^*$:
\begin{lemma}
  Let $(\Xi^-,\Xi^+)$ be a Dedekind cut on $M^*$, meaning specifically
  that $M^*$ is the disjoint union of $\Xi^-$ and $\Xi^+$, and $\Xi^-
  < \Xi^+$.  Let $\rho^+$ and $\rho^-$ be two automorphisms of $M$.
  Suppose that $\rho^+,\rho^-,$ and $\iota$ each preserve the Dedekind
  cut (for example, $\iota(\Xi^-) = \Xi^-$).  Then the map:
  \begin{equation*}
    \rho := (\rho^- \restriction \Xi^-) \cup (\rho^+ \restriction \Xi^+)
  \end{equation*}
  is an automorphism of $M^*$.
\end{lemma}
\begin{proof}
  By inspection, $\tilde{P}(x_0,\ldots,x_4)$ cannot hold unless the
  $x_i$ are within distance $\pi$ of each other, in which case they
  must lie entirely on one side of the Dedekind cut.  Consequently,
  the preservation of $\tilde{P}$ by $\rho$ can be checked on each
  side of the Dedekind cut in isolation.  The preservation of $\le$ and
  $\iota$ by $\rho$ are similar or easier.
\end{proof}

Let $\Xi^\pm$ be the Dedekind cut just beyond the end of $M$, so $\Xi^+$
is the set of upper bounds of $M$ in $M^*$.  This Dedekind cut is
fixed by $\iota$, $\tau_1$ and $\tau_2$, because each of these maps sends $M$ to $M$ setwise.  By $\aleph_1$-saturation, $\Xi^+$ is
non-empty.  For $i = 1,2$, let $\sigma_i$ be the automorphism obtained
by gluing the identity map on $\Xi^-$ with $\tau_i$ on $\Xi^+$.  So
$\sigma_i$ fixes $\Xi^-$ pointwise, and agrees with $\tau_i$ on
$\Xi^+$.  Thus,
\begin{enumerate}
\item The $\sigma_i$ fix $M \subseteq \Xi^-$ pointwise.
\item Every element of the home sort fixed by $\sigma_1$ is fixed by
  $\sigma_2$.
\item There is an element $a \in \Xi^+$ fixed by $\tau_1$ and $\sigma_1$, as the fixed points of $\tau_1$ are cofinal.
\item For this element $a$, the maps $\sigma_i$ and $\tau_i$ agree on $X_a$.
  Consequently $\sigma_1$ fixes every element of $X_a/E_a$, and
  $\sigma_2$ fixes no element of $X_a/E_a$.
\end{enumerate}
By Lemma~\ref{ei-fail-test}, it follows that the $aM$-definable
quotient $X_a/E_a$ is not $aM$-definably eliminated.

Now let $X$ be
the 0-definable set of all triples of real elements
\begin{equation*}
  \{(x,y,z) : x < y < z < \iota(x)\}
\end{equation*}
and let $E$ be the 0-definable relation
\begin{align*}
  (x,y,z)E(x',y',z') & \iff x = x' \wedge \tilde{P}(x,y,z,y',z') \\
	   & \iff x = x'  \text{ and }  \\ & \cot(z - x) - \cot(y - x) = \cot(z' - x') - \cot(y' - x') 
\end{align*}
Then $E$ is an equivalence relation on $X$.
For any $a$, there is an $a$-definable injection $X_a \hookrightarrow X$ given by $(x,y) \mapsto
(a,x,y)$, and this induces an $a$-definable injection $X_a/E_a
\hookrightarrow X/E$.
\begin{proposition}
  In the structure $M$, the quotient $X/E$ is not $M$-definably
  eliminated.
\end{proposition}
\begin{proof}
  Otherwise, there would be an $M$-definable injection from $X/E$ into
  $M^k$.  In the elementary extension $M^*$ considered above, this would yield an $M$-definable injection from $X/E$ into $(M^*)^k$.  Above, we found an element $a \in M^*$ such that the $aM$-definable quotient $X_a/E_a$ is not $aM$-definably eliminated.  However, the composition
	\begin{equation*}
	X_a/E_a \hookrightarrow X/E \hookrightarrow (M^*)^k
	\end{equation*}
	is an $aM$-definable injection that eliminates the quotient $X_a/E_a$, a contradiction.
\end{proof}

\section{Good quotient topologies} \label{sec:positive}
We next turn our attention to Theorem~\ref{main-theorem}, which shows
that quotient topologies on definable quotients are sometimes
well-behaved.  We begin by discussing the topological tools that will
be used in the proof.

\subsection{Definable topologies and definable compactness}
Work inside a model-theoretic structure $M$.  Recall that a topology
on an interpretable set $X$ is \emph{definable} if some definable
family of subsets of $X$ constitutes a basis for the topology.
Typical examples include:
\begin{enumerate}
\item The order topology on any ordered structure
\item The standard topology on $M^n$ for any o-minimal structure $M$.
\item The valuation topology on any model of ACVF or $p$CF
  ($p$-adically closed fields).
\item The discrete topology on any structure
\end{enumerate}

\begin{remark}
  Let $X$ and $Y$ be definable topological spaces.
  \begin{enumerate}
  \item The subspace topology on any definable subset of $X$ is a
    definable topology.
  \item The sum and products topologies on $X \coprod Y$ and $X \times
    Y$ are definable.
  \item If $D$ is a definable subset of $X$, then $\overline{D}$ is
    definable.
  \item As $D$ ranges over a definable family of subsets of $X$,
    $\overline{D}$ ranges over a definable family.
  \end{enumerate}
\end{remark}
In definable topological spaces, there are notions of ``definable
connectedness'' and ``definable compactness'' behaving similarly to
normal connectedness and compactness.  Here we will only deal with
definable compactness.\footnote{Definition \ref{d3.2} does not appear in the
  literature, except for some slides and unpublished notes of
  Fornasiero \cite{fornasiero}.}\footnote{There is an alternative notion of ``definable compactness'' in the o-minimal setting, due to Peterzil and Steinhorn \cite{pz-definable-space}.  The Peterzil-Steinhorn definition uses completable curves, and is primarily geared for the setting of ``definable spaces.'' In our terminology, Peterzil-Steinhorn definable spaces are definable topological spaces covered by finitely many open sets, each of which is homeomorphic to a definable subset of $M^n$ with the induced subspace topology.  Since exotic interpretable sets never admit such coverings, we do not use the Peterzil-Steinhorn theory.  It is unclear whether our notion of definable compactness (Definition~\ref{d3.2}) agrees with Peterzil and Steinhorn's definition, when restricted to definable spaces.}

Say that a partial order $(\le, P)$ is downwards-directed if every
finite non-empty subset of $P$ has a lower bound, and upwards-directed
if every finite non-empty subset of $P$ has an upper bound.  Recall
that a topological space is compact if every downwards-directed family
of non-empty closed sets has non-empty intersection.
\begin{definition}\label{d3.2}
  A definable topological space $X$ is \emph{definably compact} if
  $\bigcap \mathcal{F}$ is non-empty, for every definable family
  $\mathcal{F}$ of non-empty closed subsets of $X$ that is
  downwards-directed with respect to inclusion.

  More generally, a definable subset $D \subseteq X$ is said to be
  \emph{definably compact} if the induced subspace topology on $D$ is
  definably compact.
\end{definition}
\begin{example}
  ~
  \begin{enumerate}
  \item The order topology on $(\Rr,<)$ is not definably compact due
    to the family of half-infinite intervals $[a,+\infty)$, which has
      empty intersection in spite of being a downwards directed family
      of closed non-empty sets.
  \item In contrast, $[0,1]$ is definably compact in $(\Rr,<)$,
    because it is compact.
  \item The closed interval $[0,1]$ is definably compact in
    $(\Qq,\le)$, because this is elementarily equivalent to the
    previous example.
  \item The discrete topology on any pseudofinite or NSOP set is
    definably compact, because downwards-directed families of subsets
    must have minima.  For example, the discrete topology on a
    pseudofinite field or an algebraically closed field is definably
    compact.
  \item In $\Qq_p$, the ring of integers $\Zz_p$ is definably compact
    in the valuation topology, because it is compact.  More generally,
    the ring of integers in a $p$-adically closed field is definably
    compact in the valuation topology.
  \item If $K$ is a pseudofinite field, then the ring $K[[t]]$ is
    definably compact with respect to the valuation topology (i.e.,
    the $(t)$-adic topology), because it is elementarily equivalent to
    an ultraproduct of the previous examples.
  \item One can show that $\Cc[[t]]$ is definably compact in the
    valuation topology, using the fact that the residue field is a
    pure algebraically closed field.
  \end{enumerate}
\end{example}

We now verify that many of the familiar properties
of compactness hold for definable compactness.
(Fornasiero has independently made these 
observations in \cite{fornasiero}.)
\begin{lemma} \label{kpct-img}
  Let $f : X \to Y$ be a definable continuous map between two
  definable topological spaces.  Then $f(K)$ is definably compact for
  any definable compact set $K \subseteq X$.
\end{lemma}
\begin{proof}
  Replacing $X$ and $Y$ with $K$ and $f(K)$, we may assume $K = X$ and
  $f$ is surjective.  Let $\mathcal{F}$ be a downwards-directed
  definable family of non-empty closed subsets of $Y$.  As $f$ is
  surjective, $f^{-1}(F)$ is a non-empty closed subset of $X$ for each
  $F \in \mathcal{F}$.  Moreover, the map
  \begin{equation*}
    F \mapsto f^{-1}(F)
  \end{equation*}
  is order-preserving, so the family
  \begin{equation*}
    \{ f^{-1}(F) : F \in \mathcal{F}\}
  \end{equation*}
  is downwards-directed.  This family is a definable family, so by
  definable compactness on $X$, there is some $x_0 \in X$ such that
  \begin{equation*}
    x_0 \in f^{-1}(F) \qquad \forall F \in \mathcal{F}
  \end{equation*}
  or equivalently,
  \begin{equation*}
    f(x_0) \in F \qquad \forall F \in \mathcal{F}
  \end{equation*}
  Thus $\bigcap \mathcal{F}$ is non-empty, proving definable
  compactness of $Y$.
\end{proof}

\begin{lemma} \label{obv-facts}
  ~
  \begin{enumerate}
  \item \label{obv-facts-closed} If $K$ is a definably compact
    definable topological space, and $F \subseteq K$ is a closed
    subset, then $F$ is definably compact itself.
  \item If $K_1$ and $K_2$ are definably compact, so is $K_1 \cup
    K_2$.
  \end{enumerate}
\end{lemma}
\begin{proof}
  \begin{enumerate}
  \item Any downwards-directed definable family of closed non-empty
    subsets of $F$ is also a downwards-directed definable family of
    closed non-empty subsets of $K$, so definable compactness directly
    transfers.
  \item Let $\mathcal{F}$ be a downwards-directed definable family of
    closed subsets of $K_1 \cup K_2$.  Suppose $\bigcap \mathcal{F} =
    \emptyset$.  We will show $\emptyset \in \mathcal{F}$.

    If $F$ is a closed definable subset of $K_1 \cup K_2$, then $F
    \cap K_1$ and $F \cap K_2$ are closed subsets of $K_1$ and $K_2$.
    The maps
    \begin{align*}
      F & \mapsto F \cap K_1 \\
      F & \mapsto F \cap K_2
    \end{align*}
    are order-preserving, so the families
    \begin{align*}
      \mathcal{F}_1 & := \{F \cap K_1 : F \in \mathcal{F}\} \\
      \mathcal{F}_2 & := \{F \cap K_2 : F \in \mathcal{F}\}
    \end{align*}
    are also downwards-directed definable families of closed sets.  Note that
    \begin{equation*}
      \bigcap \mathcal{F}_i \subseteq \bigcap \mathcal{F} = \emptyset
    \end{equation*}
    for $i = 1, 2$.  Consequently $\emptyset \in \mathcal{F}_i$ for $i
    = 1,2$, meaning that there are $F_1, F_2 \in \mathcal{F}$ such
    that
    \begin{equation*}
      F_i \cap K_i = \emptyset
    \end{equation*}
    for $i = 1, 2$.  By downward-directedness, there is some $F_3 \in
    \mathcal{F}$ such that $F_3 \subseteq F_1 \cap F_2$.  Then
    \begin{equation*}
      F_3 \cap (K_1 \cup K_2) = (F_3 \cap K_1) \cup (F_3 \cap K_2)
      \subseteq (F_1 \cap K_1) \cup (F_2 \cap K_2) = \emptyset \cup
      \emptyset = \emptyset.
    \end{equation*}
    So $\emptyset \in \mathcal{F}$.
  \end{enumerate}
\end{proof}

Say that a definable map $f : X \to Y$ of definable topological spaces
is \emph{definable closed} if $f(D)$ is closed for every closed
definable subset $D \subseteq X$.  This is a weaker condition than
being a closed map: for example, in the structure $(\Qq,\le)$, the
projection $\Qq \times [0,1] \to \Qq$ is not closed\footnote{Take a sequence $a_1, a_2, \ldots$ of rational numbers in $[0,1]$ converging to an irrational number.  If $S = \{(1/n,a_n) : n \in \Nn\}$, then $S$ is closed (as a subset of $\Qq \times \Qq$), but its projection onto the first coordinate is not closed.}, but is
definably closed (by Lemmas~\ref{pseudo-closed-map} and
\ref{o-minimal-kpct}).
\begin{lemma}\label{pseudo-closed-map}
  Let $X$ and $K$ be definable topological spaces, with $K$ definably
  compact.  Consider the product topology on $X \times K$ and let $\pi
  : X \times K \twoheadrightarrow X$ be the projection.  Then $\pi$ is
  definably closed.
\end{lemma}
\begin{proof}
  Suppose $F$ is a closed subset of $X \times K$ and $x_0 \in X
  \setminus \pi(F)$.  We will show $x_0 \notin \overline{\pi(F)}$, so
  that $\pi(F) = \overline{\pi(F)}$.  For each open neighborhood $N$
  of $x_0$, let
  \begin{equation*}
    N^\dag := \{k \in K : \text{ there is an open neighborhood $U$ of
      $k$ such that } (N \times U) \cap F = \emptyset\}
  \end{equation*}
  Note that $N^\dag$ is open and map $N \mapsto N^\dag$ is
  order-reversing.  Let $\mathcal{N}$ be a definable neighborhood
  basis of $x_0$, and let
  \begin{equation*}
    \mathcal{N}^\dag := \{N^\dag : N \in \mathcal{N}\}
  \end{equation*}
  Because $\mathcal{N}$ is downwards-directed, $\mathcal{N}^\dag$ is
  upwards-directed.

  Furthermore, $\bigcup \mathcal{N}^\dag = K$.  Indeed, if $k$ is any
  element of $K$, then $(x_0,k) \notin F$, by choice of $x_0$, so some
  open neighborhood $N \times U$ of $(x_0,k)$ avoids $F$, as $F$ is
  closed.

  So $\mathcal{N}^\dag$ is an upwards-directed definable family of
  open subsets of $K$, whose union is all of $K$.  By definable
  compactness, $K \in \mathcal{N}^\dag$.  So there is some $N \in
  \mathcal{N}$ with $N^\dag = K$, implying that $(N \times K) \cap F =
  \emptyset$, and thus $N \cap \pi(F) = \emptyset$.  Thus we have
  produced an open neighborhood $N$ of $x_0$ disjoint from $\pi(F)$,
  showing that $x_0 \notin \overline{ \pi(F)}$.  As $x_0$ was an
  arbitrary point not in $\pi(F)$, it follows that $\pi(F)$ is closed.
\end{proof}

\begin{proposition} \label{kpct-prods}
  Let $X$ and $Y$ be definably compact definable topological spaces.
  Then $X \times Y$ is definably compact.
\end{proposition}
\begin{proof}
  We may assume $X$ and $Y$ are non-empty.  Let $\pi : X \times Y \to
  X$ denote the projection.  Suppose $\mathcal{F}$ is a
  downwards-directed definable family of non-empty closed subsets of
  $X \times Y$.  For each $F \in \mathcal{F}$, the projection $\pi(F)$
  is closed, by Lemma~\ref{pseudo-closed-map}, and obviously
  non-empty.  Furthermore, the map $F \mapsto \pi(F)$ is
  order-preserving.  Consequently, the family
  \begin{equation*}
    \{\pi(F) : F \in \mathcal{F}\}
  \end{equation*}
  is a downwards-directed definable family of closed non-empty subsets
  of $X$.  By definable compactness of $X$, we may find some $x_0$
  such that
  \begin{equation*}
    x_0 \in \pi(F) \qquad \forall F \in \mathcal{F}
  \end{equation*}
  Equivalently, $F \cap (\{x_0\} \times Y)$ is non-empty for every $F
  \in \mathcal{F}$.  Note that $\{x_0\} \times Y$ is definably compact
  (as a subset of $X \times Y$) because it is definably homeomorphic
  to $Y$.  The family
  \begin{equation*}
    \{F \cap (\{x_0\} \times Y) : F \in \mathcal{F}\}
  \end{equation*}
  is a definable family of non-empty closed subsets of $\{x_0\} \times
  Y$, and it is downwards-directed because the map
  \begin{equation*}
    F \mapsto F \cap (\{x_0\} \times Y)
  \end{equation*}
  is order-preserving.  By definable compactness of
  $\{x_0\} \times Y$, we can find some $(x_0,y_0)$ which is in every
  $F$, showing that $\bigcap \mathcal{F}$ is non-empty.
\end{proof}

\begin{lemma} \label{kpct-f}
  Let $X$ be a definable topological space that is Hausdorff, and let
  $K$ be a definably compact subset.  Then $K$ is closed.
\end{lemma}
\begin{proof}
  Otherwise, fix $x_0 \in \partial K$.  Let $\mathcal{N}$ be a
  definable neighborhood basis of $x_0$.  The family $\mathcal{N}$ is
  downwards directed, and the map
  \begin{equation*}
    N \mapsto \overline{N} \cap K
  \end{equation*}
  is order-preserving, so the family
  \begin{equation*}
    \{\overline{N} \cap K : N \in \mathcal{N}\}
  \end{equation*}
  is a downwards-directed definable family of closed subsets of $K$.
  Furthermore, none of the sets $\overline{N} \cap K$ is empty, because
  $x \in \partial K$, so each $N$ intersects $K$.  By definable
  compactness, there is some $x_1$ such that
  \begin{equation*}
    x_1\in \overline{N} \cap K \qquad \forall N \in \mathcal{N}
  \end{equation*}
  Then $x_1\in K$, so $x_1 \ne x_0$.  By the Hausdorff property, some
  open neighborhood $N$ of $x_0$ satisfies $x_1 \notin \overline{N}$.
  Shrinking $N$ a little, we may assume $N \in \mathcal{N}$, and
  obtain a contradiction.
\end{proof}

\begin{lemma} \label{o-minimal-kpct}
  If $M$ is an o-minimal structure, then any closed interval
  $[c,d] \subset M^1$ is definably compact in the order topology.
\end{lemma}
\begin{proof}
  Let $\mathcal{F}$ be a downwards-directed definable family of
  non-empty closed subsets of $[c,d]$.  O-minimality ensures that
  $\max F$ exists for each $F \in \mathcal{F}$.  Let
  \begin{equation*}
    S = \{\max F : F \in \mathcal{F}\}
  \end{equation*}
  This is a definable subset of [c,d], so $s_0 = \inf S$ exists.  We claim that
  $s_0 \in F$ for all $F \in \mathcal{F}$.

  Otherwise, by closedness of the $F$'s, there must be some open
  interval $(a,b)$ around $s_0$, and some $F_0 \in \mathcal{F}$, such
  that $(a,b) \cap F_0 = \emptyset$.  Since $s$ is the infimum of $S$,
  it must be in the closure of $S$, so $S$ must intersect $(a,b)$.  In
  particular, there must be some $s_1 \in S \cap (a,b)$.  By
  definition of $S$, there is some $F_1 \in \mathcal{F}$ such that
  $s_1 = \max F_1$.  By downwards directedness, there is some $F_2 \in
  \mathcal{F}$ such that $F_2 \subseteq F_0 \cap F_1$.  Then
  \begin{equation*}
    F_2 \subseteq F_1 \subseteq (-\infty,s_1] \subseteq (-\infty,b)
  \end{equation*}
  because $s_1 = \max F_1$ and $s_1 < b$.  Additionally,
  \begin{equation*}
    F_2 \cap (a,b) \subseteq F_0 \cap (a,b) = \emptyset.
  \end{equation*}
  Combining these, we see that $F_2 \subseteq
  (-\infty,a]$. Consequently, $\max F_2 \le a < s_0$, contradicting
  the choice of $s_0$.
\end{proof}

The next proposition shows that our definition of definable
compactness agrees with the standard one in o-minimal structures.
\begin{proposition}
  Let $(M,<,\ldots)$ be an o-minimal structure.  In the standard
  topology on $M^n$ the definably compact sets are exactly the closed
  bounded sets.
\end{proposition}
\begin{proof}
  Let $X \subseteq M^n$ be definable.  

  First suppose that $X$ is closed and bounded.  Then $X \subseteq
  [a,b]^n$ for some $a, b \in M$.  By Lemma~\ref{o-minimal-kpct},
  $[a,b]$ is definably compact, and by Proposition~\ref{kpct-prods},
  $[a,b]^n$ is definably compact.  Finally, the closed subset $X$ of
  $[a,b]^n$ is compact by
  Lemma~\ref{obv-facts}(\ref{obv-facts-closed}).

  Next suppose $X$ is not bounded.  Then for every $a \le b$, the
  intersection
  \begin{equation*}
    X \cap \left((-\infty,a] \cup [b,+\infty)\right)^n
  \end{equation*}
  is non-empty.  The family of all such intersections is a definable
  downwards-directed family of closed non-empty subsets of $X$.
  However, its intersection is empty, so $X$ is not definably compact.

  Finally, suppose $X$ is not closed.  Then $X$ fails to be definably
  compact by Lemma~\ref{kpct-f}, because the standard topology on
  $M^n$ is Hausdorff.
\end{proof}

\begin{lemma} \label{no-space-filling-injections}
  Let $f : X \to Y$ be a definable continuous map from a definable
  topological space $X$ to a definable topological space $Y$.  If $X$
  is definably compact, $Y$ is Hausdorff, and $f$ is injective, then
  $f$ is a homeomorphism onto its image.
\end{lemma}
\begin{proof}
  Shrinking $Y$, we may assume $f$ is a bijection.  For any definable
  subset $D \subseteq X$, $D$ is closed in $X$ if and only if $f(D)$ is
  closed in $Y$.  Indeed, if $f(D)$ is closed, then $D = f^{-1}(f(D))$
  is closed by continuity, and conversely, if $D$ is closed, then $D$
  is compact by Lemma~\ref{obv-facts}(\ref{obv-facts-closed}), $f(D)$
  is compact by Lemma~\ref{kpct-img}, and $f(D)$ is closed by
  Lemma~\ref{kpct-f}.

  Equivalently, a definable subset $D \subseteq X$ is open in $X$ if and
  only if $f(D)$ is open in $Y$.  Because $X$ and $Y$ have definable bases of opens, this is enough to ensure that $f$ is a homeomorphism.
\end{proof}

\subsection{Quotient topologies and open maps}
\label{sec:o-eq}

Recall (\cite{lou-o-minimality} \S6.4) that a surjective continuous map
$f : X \to Y$ is an \emph{identifying map} if
\begin{equation*}
  f^{-1}(U) \text{ is open } \implies U \text{ is open}
\end{equation*}
for all $U \subseteq X$.  For a fixed topological space $X$, the
identifying maps out of $X$ are exactly the maps of the form $X
\twoheadrightarrow X/E$ where $X/E$ has the quotient topology.

Note that surjective open maps are identifying.  Say that an
equivalence relation $E$ on a topological space is an \emph{open
  equivalence relation} if the quotient map $X \twoheadrightarrow X/E$
is an open map.

For $D$ a subset of $X$, let $D^E$ denote the union of $E$-equivalence
classes intersecting $D$.  We will call this the \emph{$E$-closure} of
$D$.  An equivalence relation $E$ is an open equivalence relation
exactly if the $E$-closure of any open set is open.

We are interested in open equivalence relations because they ensure
definability of the quotient topology, in a model-theoretic setting:

\begin{lemma} \label{quot-def}
  Let $X$ be an interpretable set with a definable topology.  Let $E$
  be a definable open equivalence relation on $X$.  Then the quotient
  topology on $X/E$ is a definable topology.
\end{lemma}
\begin{proof}
  Let $f : X \twoheadrightarrow X/E$ be the quotient map, which is a
  surjective open map.  Note that the open subsets of $X/E$ are
  exactly the sets of the form $f(U)$ for $U$ an open in $X$.  Let
  $\mathcal{B}$ be a definable basis of opens for $X$.  Then
  \begin{equation*}
    \{f(B) : B \in \mathcal{B}\}
  \end{equation*}
  is a definable basis for the topology on $X/E$.
\end{proof}

In the proof of Theorem~\ref{main-theorem}, we will prove that certain
properties hold generically, and then shrink to open sets on which
these properties hold.  Open equivalence relations help ensure that
the topology does not change too much when we pass to open subsets:
\begin{lemma} \label{niceness-of-o-eq}
  Let $X$ be a topological space and $E$ be an open equivalence
  relation on $X$.  Let $X'$ be an open subset of $X$.
  \begin{enumerate}
  \item The restriction of $E$ to $X'$ is an open equivalence relation.
  \item There are two topologies on $X'/E$, the subspace topology (as
    a subset of $X/E$) and the quotient topology (as a quotient of
    $X'$).  These two topologies agree.
  \item The map $X'/E \hookrightarrow X/E$ is an open embedding.
  \end{enumerate}
\end{lemma}
\begin{proof}
  View $X'/E$ as topological space via the subspace topology.  The map
  $f : X \to X/E$ is an open map, so $f(X') = X'/E$ is an open subset
  of $X/E$, proving (3).

  The top and right maps in the following commutative diagram are open
  maps, so their composition is also an open map.
  \begin{equation*}
    \xymatrix{ X' \ar@{^{(}->}[r] \ar@{->>}[d] & X \ar@{->>}[d] \\
      X'/E \ar@{^{(}->}[r] & X/E}
  \end{equation*}
	Because the diagonal is an open map and the bottom map $X'/E \hookrightarrow X/E$ is a continuous injection, it follows that the left map $X' \twoheadrightarrow X'/E$ is an open map.  Open surjective maps are identifying maps, so
  $X'/E$ has the quotient topology from $X'$, proving (2).  Having
  shown that $X'/E$ has the quotient topology, (1) means precisely
  that $X' \twoheadrightarrow X'/E$ is an open map, which we showed.
\end{proof}

\subsection{Proof of Theorem \ref{main-theorem}}
In this section, we will work inside a fixed o-minimal structure $M$.
If $X \subseteq Y$ is an inclusion of interpretable sets, we will say
that $X$ is a \emph{full subset} of $Y$ if $\dim(Y \setminus X) < \dim
Y$.

We will prove the following refinement of Theorem~\ref{main-theorem}:
\begin{theorem} \label{better-theorem}
  Let $X \subseteq M^n$ be a definable set, and $E$ be a definable
  equivalence relation on $X$.  There is a definable full open subset
  $X'$ of $X$ such that if $E'$ is the restriction of $E$ to $X$, then
  $E'$ is an open equivalence relation on $X'$ (in the sense of
  \S\ref{sec:o-eq}), and the quotient topology on $X'/E'$ is Hausdorff and
  locally Euclidean.
\end{theorem}
The requirement that $X'$ is a full open subset of $X$ is exactly
equivalent to conditions \ref{cond1} and \ref{cond2} of
Theorem~\ref{main-theorem}.  Lemma~\ref{quot-def} ensures that the
quotient topology $X'/E'$ is definable, and
Lemma~\ref{niceness-of-o-eq} ensures that the final condition
\ref{cond4} of Theorem~\ref{main-theorem} holds.

For the proof of Theorem~\ref{better-theorem}, we may assume that $X$
and $E$ are 0-definable, by naming parameters otherwise.  We may also
assume that the language is countable (by passing to a reduct
otherwise).

In proving Theorem~\ref{better-theorem}, we may replace $M$ with an
$\aleph_1$-saturated elementary extension.  The
topological properties other than local Euclideanity are all
expressible by first-order sentences.  In $\aleph_1$-saturated models,
local Euclideanity implies \emph{uniform local Euclideanity}, local
Euclideanity witnessed by charts of bounded complexity.  And then
uniform local Euclideanity can be expressed as a disjunction of
first-order sentences, so it descends from the elementary extension to
the original structure.

Thus, in what follows, we will assume that the language is countable,
and that the ambient o-minimal structure is $\aleph_1$-saturated.  For
a 0-definable or 0-interpretable set $D$, we will say that an element
$a \in D$ is \emph{generic (in $D$)} if $\dim(a/\emptyset) = \dim D$.

The following lemma contains the main tricks we will use in the proof:
\begin{lemma} \label{tricks}
  Let $X \subseteq M^n$ be a 0-definable set.  Working inside the
  definable topological space $X$,
  \begin{enumerate}
  \item \label{trick-dim} $\dim \partial D < \dim D$ for any non-empty
    definable set $D$, where the frontier is taken \emph{inside} $X$.
  \item \label{trick-generic} Let $P$ be a subset of $X$ which is
    0-definable or 0-ind-definable.  Suppose that $P$ contains every
    generic element of $X$.  Then $P$ contains a full open
    0-definable subset $X'$ of $X$.
  \item \label{trick-indy} Let $S$ be any countable set, and let $a$
    be an element of $X$. The collection of definable open
    neighborhoods $B$ of $a$ such that
    \begin{equation*}
      \ulcorner B \urcorner \forkindep aS
    \end{equation*}
    form a neighborhood basis of $a$.
  \end{enumerate}
\end{lemma}
\begin{proof}
  \begin{enumerate}
  \item The frontier of $D$ within $X$ is smaller than the
    frontier of $D$ within the ambient space $M^n$, and for $M^n$ this
    fact is \cite{lou-o-minimality} Theorem 4.1.8.
  \item Note that $X \setminus P$ is type-definable over $\emptyset$
    and contains only elements of rank less than $\dim X$ over
    $\emptyset$.  Thus
    \begin{equation*}
      X \setminus P \subseteq D
    \end{equation*}
    for some 0-definable $D$ with
    \begin{equation*}
      \dim \overline{D} = \dim D < \dim X,
    \end{equation*}
    and then we can take $X' = X \setminus \overline{D}$.
  \item We can take $B$ of
    the form
    \begin{equation*}
      X \cap \prod_{i = 1}^n ]b_i,c_i[
    \end{equation*}
    where the $b_i$ and $c_i$ are close to $a$ but independent from
    everything in sight.
  \end{enumerate}  
\end{proof}


We break the proof of Theorem~\ref{better-theorem} into three steps,
which are the next three propositions.

\begin{proposition} \label{part1}
  Let $X \subseteq M^n$ be 0-definable and $E$ be a 0-definable
  equivalence relation on $X$.  There is a 0-definable full open
  subset $X' \subseteq X$ on which the restriction $E \restriction X'$
  is an open equivalence relation.
\end{proposition}
\begin{proof}
  Recall from \S \ref{sec:o-eq} that for $S \subseteq X$, the
  \emph{$E$-closure} of $S$, denoted $S^E$, is the union of all
  $E$-equivalence classes that intersect $S$.

  Say that a point $a \in X$ is \emph{nice} if for every $b \in
  \{a\}^E$, and every neighborhood $B$ of $b$,
  \begin{equation*}
    a \in \ter(B^E).
  \end{equation*}
  Note that we could equivalently restrict to basic open neighborhoods, so ``niceness'' is definable.

  \begin{claim}
    Every generic element of $X$ is nice.  That is, if $a\in X$ and
    $\dim(a/\emptyset) = \dim X$, then $a$ is nice.
  \end{claim}
  \begin{proof}
    Suppose otherwise.  Let $a$ be generic, $b$ be another point such
    that $aEb$ holds, and $B$ be an open neighborhood of $b$ such that
    $a \notin \ter(B^E)$.  Shrinking $B$, we may assume by
    Lemma~\ref{tricks}(\ref{trick-indy}) that
    \begin{equation*}
      \ulcorner B \urcorner \forkindep ab.
    \end{equation*}
    As $aEb$ and $b \in B$, we see $a \in B^E$.  But by assumption, $a
    \notin \ter(B^E)$, and so $a \in \partial (X \setminus
    B^E)$.  Then
    \begin{equation*}
      \dim(a/\ulcorner B \urcorner) \le \dim \partial (X \setminus
      B^E) < \dim (X \setminus B^E) \le \dim X = \dim(a/\emptyset),
    \end{equation*}
    contradicting the independence of $a$ and $\ulcorner B \urcorner$.
  \end{proof}

  The set of nice points is a 0-definable subset of $X$.  By
  Lemma~\ref{tricks}(\ref{trick-generic}), there is a 0-definable full
  open subset $X' \subseteq X$ consisting only of nice points.  Let
  $E'$ be the restriction of $X$ to $E$.  Note that a subset of $X'$
  is open as a subset of $X'$ if and only if it is open as a subset of
  $X$.  So we can talk unambiguously about ``open'' sets.

  We claim that $E'$ is an open equivalence relation on $X'$.
  Otherwise, there is an open subset $U$ of $X'$ such that $U^{E'}$ is
  not open.  Take $a \in U^{E'} \setminus \ter(U^{E'})$, and choose a
  point $b \in U$ such that $aE'b$ holds.

  In $X$, we have two $E$-equivalent points $a, b$ and an open
  neighborhood $U$ of $b$.  As $a$ is nice, $a \in
  \ter(U^E)$, meaning that there is a neighborhood $V$ of $a$
  in $X$ such that every point of $V$ is connected via $E$ to a point
  in $U$.  Shrinking $V$, we may assume $V \subseteq X'$.  Then $V$
  and $U$ are in $X'$, so every element of $V$ is connected via $E'$
  to some element of $U$, meaning that $V \subseteq U^{E'}$.  Now $V$
  witnesses that $a \in \ter(U^{E'})$, a contradiction.
\end{proof}

\begin{proposition} \label{part2}
  Let $X \subseteq M^n$ be 0-definable and $E$ be a 0-definable open
  equivalence relation on $X$.  There is a 0-definable full open
  subset $X' \subseteq X$ such that $X'/E$ is Hausdorff.
\end{proposition}
Here, the topology on $X'/E$ is either the quotient topology from the
subspace topology on $X'$, or the subspace topology from the quotient
topology on $X/E$.  These two topologies agree by
Lemma~\ref{niceness-of-o-eq}.
\begin{proof}
  Let $\pi : X \twoheadrightarrow X/E$ denote the quotient map.
  \begin{claim}\label{cl3.19}
    Let $a$ and $b$ be two generic elements of $X$ (perhaps not jointly
    generic).  If $a$ and $b$ are in different $E$-equivalence classes,
    then there exist basic open neighborhoods $N_1$ and $N_2$ around $a$
    and $b$, respectively, such that $\pi(N_1) \cap \pi(N_2) =
    \emptyset$.
  \end{claim}
  \begin{proof}
    We claim that $a \notin \overline{\{b\}^E}$.  Suppose otherwise.
    Then $a \in \partial( \{b\}^E)$.  Let $c$ be an element of $\{b\}^E$
    of maximal rank over $\pi(b)$.  Then
    \begin{align*}
      \dim(a/\emptyset) & \le \dim(a \pi(b)/\emptyset) = \dim(a/\pi(b)) +
      \dim(\pi(b)/\emptyset) 
      \\ & \le \dim(\partial (\{b\}^E)) +
      \dim(\pi(b)/\emptyset) < \dim(\{b\}^E) + \dim(\pi(b)/\emptyset) 
      \\ & =
      \dim(c/\pi(b)) + \dim(\pi(b)/\emptyset) = \dim(c \pi(b)/\emptyset)
      \\ & = \dim(c/\emptyset) \le \dim X,
    \end{align*}
    contradicting the fact that $a$ is generic.

    So $a$ is not in the closure of $\{b\}^E$, and therefore some open
    neighborhood $N_1$ of $a$ is disjoint from $\{b\}^E$.  Shrinking
    $N_1$ slightly, we may assume by
    Lemma~\ref{tricks}(\ref{trick-indy}) that $\ulcorner N_1
    \urcorner$ is independent from $b$.  Because $b$ is then generic
    over $\ulcorner N_1 \urcorner$, we see that $b \notin \partial (
    N_1^E)$.  Now by choice of $N_1$, $b \notin N_1^E$.  Therefore $b
    \notin \overline{N_1^E}$.  So we can find an open neighborhood
    $N_2$ of $b$, disjoint from $N_1^E$.  The fact that $N_2$ is
    disjoint from $N_1^E$ means exactly that $\pi(N_1)$ and $\pi(N_2)$
    are disjoint.
  \end{proof}
  Let $\Sigma(x)$ be the partial type over $\emptyset$ asserting that
  $x$ is generic over $\emptyset$.  Let $D$ be the set of pairs $(x,y)
  \in X \times X$ such that either $\pi(x) = \pi(y)$ or there exist
  neighborhoods $N_1$ of $x$ and $N_2$ of $y$ such that $\pi(N_1)$ and
  $\pi(N_2)$ are disjoint.  Note that $D$ is 0-definable.  By Claim~\ref{cl3.19},
  \begin{equation*}
    \Sigma(x) \wedge \Sigma(y) \vdash (x,y) \in D
  \end{equation*}
  By compactness, there is some 0-definable set $X'$ such that
  \begin{equation*}
    \Sigma(x) \vdash x \in X'
  \end{equation*}
  and $X' \times X' \subseteq D$.  Shrinking $X'$ a little, we may
  assume $X'$ is a full open subset of $X$, as in the proof of
  Lemma~\ref{tricks}(\ref{trick-generic}).

  Let $E'$ be the restriction of $E$ to $X'$.  By
  Lemma~\ref{niceness-of-o-eq}, $X'/E'$ is an open subset of $X/E$,
  and $E'$ is an open equivalence relation on $X'$.  We claim that
  $X'/E'$ is Hausdorff.

  Let $a_0, b_0$ be two distinct elements of $X'/E'$, and let $a$ and
  $b$ be lifts of $a_0$ and $b_0$ to $X'$. By choice of $X'$, the pair
  $(a,b)$ is in $D$.  As $\pi(a) = a_0 \ne b_0 = \pi(b)$, $a$ and $b$
  are not $E$-equivalent.  By definition of $D$, there exist
  neighborhoods $N_1$ and $N_2$ in $X$, around $a$ and $b$, such that
  $\pi(N_1)$ is disjoint from $\pi(N_2)$. Because $\pi : X \to X/E$ is
  an open map, $\pi(N_1)$ is a neighborhood of $a_0$, and $\pi(N_2)$
  is an open neighborhood of $b_0$.  Therefore, $a_0$ and $b_0$ can be
  separated by open neighborhoods in $X/E$, hence also in $X'/E'$.
\end{proof}

\begin{proposition} \label{part3}
  Let $X \subseteq M^n$ be 0-definable and $E$ be a 0-definable open
  equivalence relation on $X$, with $X/E$ Hausdorff.  Then there is a
  0-definable full open subset $X' \subseteq X$ such that $X'/E$ is
  locally Euclidean.
\end{proposition}
In the proposition, note that the topologies on $X/E$ and $X'/E$ are definable, thanks to
Lemma~\ref{quot-def}.
\begin{proof}
  Let $\pi : X \to X/E$ be the quotient map.
  \begin{claim}
    It suffices to show that $X/E$ is Euclidean at $\pi(a)$ for every
    generic $a \in X$.
  \end{claim}
  \begin{proof}
    The set of $a \in X$ such that local Euclideanity holds at
    $\pi(a)$ is \emph{ind-definable} over $\emptyset$.  By Lemma~\ref{tricks}(\ref{trick-generic}), if this set
    includes every generic of $X$, then there must be a 0-definable
    full open subset $X'$ of $X$ such that local Euclideanity holds at
    $\pi(a)$ for all $a \in X'$.  By Lemma~\ref{quot-def} the map of
    quotient spaces $X'/E \hookrightarrow X/E$ is an open embedding.
    Therefore, $X'/E$ is also Euclidean at $\pi(a)$, for every $a \in
    X'$.  In other words, $X'/E$ is locally Euclidean.
  \end{proof}
  So assume that $a \in X$ is generic.  Let $e = \pi(a)$ be the image
  of $a$ in $X/E$.  We will show that $X/E$ is Euclidean at $e$, i.e.,
  that some neighborhood of $e$ is definably homeomorphic to an open
  subset of $M^k$ for some $k$.


  Choose $b$ such that $\tp(a/e) = \tp(b/e)$ and $a \forkindep_e b$.  Note that $e = \pi(b)$.
  \begin{claim}
    After re-ordering coordinates, we may write $b = b_1b_2b_3$ where
    \begin{itemize}
    \item $b_3 \in \dcl(b_1b_2)$
    \item $\dim(b_1b_2/\emptyset) = |b_1b_2|$.
    \item $b_2 \in \dcl(b_1e)$
    \item $\dim(b_1/e) = |b_1|$.
    \end{itemize}
  \end{claim}
  \begin{proof}
    In the pregeometry of definable closure over $\emptyset$, take
    $b_1b_2$ to be a maximal independent subset of $b$.  In the
    pregeometry of definable closure over $e$, take $b_1$ to be a
    maximal independent subset of $b_1b_2$.
  \end{proof}
  Let $f$ and $g$ be 0-definable functions such that
  \begin{equation*}
    b_2 = f(b_1, e)
  \end{equation*}
  \begin{equation*}
    b_3 = g(b_1, b_2)
  \end{equation*}
  Let $N$ be a countable model containing $a, b$, and let $B$ be a closed
  box with $b_2$ in its interior, such that
  \begin{equation*}
    \ulcorner B \urcorner \forkindep N
  \end{equation*}
  and such that $B$ is contained in every $N$-definable open
  neighborhood of $b_2$.

  Because $b_1, b_2$ are generic in $M^{|b_1b_2|}$, the function $g$
  is continuous on an open neighborhood of $b_1b_2$.  In particular,
  $g$ is continuous on $\{b_1\} \times B$.

  The set of $x$ such that
  \begin{equation} \label{whatever}
    x = f(b_1, \pi(b_1, x, g(b_1, x)))
  \end{equation}
  contains $b_2$, and is $b_1$-definable.  Since $b_1b_2$ is generic
  over $\emptyset$, $b_2$ is generic over $b_1$.  Therefore, $b_2$ is
  in the interior of the set of $x$ such that (\ref{whatever}) holds.
  Consequently, (\ref{whatever}) holds for $x \in B$.

  Let $h : B \to X/E$ be the map given by
  \begin{equation*}
    h(x) = \pi(b_1, x, g(b_1, x))
  \end{equation*}
  Then $h$ is continuous on $B$ (because $g$ is continuous there, and
  $\pi$ is continuous everywhere).  Furthermore, (\ref{whatever})
  shows that $h$ is injective.  By
  Lemma~\ref{no-space-filling-injections}, $B$ is homeomorphic to
  $h(B)$.  Consequently, $\ter(B)$ is homeomorphic to
  $h(\ter(B))$.

  To complete the proof of local Euclideanity around $\pi(a)$, it
  suffices to show that $\pi(a)$ is in the interior of
  $h(\ter(B))$.  The set of $x \in X$ such that
  \begin{equation*}
    \pi(x) \in h(\ter(B))
  \end{equation*}
  is definable over $b_1\ulcorner B \urcorner$, and contains $a$.  It
  suffices to show that $a$ is generic (in $X$) over $b_1 \ulcorner B
  \urcorner$.

  To see this, note that
  \begin{equation*}
    \dim(b_1/a) = \dim(b_1/e) = |b_1| = \dim(b_1/\emptyset).
  \end{equation*}
  So $b_1$ is independent from $a$.  As $\ulcorner B \urcorner$ is
  independent from everything in $N$, the sequence
  \begin{equation*}
    a, b_1, \ulcorner B \urcorner
  \end{equation*}
  is independent.  Then $\dim(a/b_1 \ulcorner B \urcorner) =
  \dim(a/\emptyset) = \dim X$, and so $a$ is generic over $b_1
  \ulcorner B \urcorner$.
\end{proof}

We now prove Theorem~\ref{better-theorem}
\begin{proof}[Proof (of Theorem~\ref{better-theorem})]
  As noted previously, we may assume the language is countable and the
  ambient model is $\aleph_1$-saturated.  By Proposition~\ref{part1},
  we may find a 0-definable full open subset $X_1 \subseteq X$ such
  that the restriction $E_1 := E \restriction X_1$ is an open
  equivalence relation.  By Proposition~\ref{part2} applied to $X_1$
  and $E_1$, there is a 0-definable full open subset $X_2 \subseteq
  X_1$ such that $X_2/E_1$ is Hausdorff.  By Proposition~\ref{part3}
  applied to $X_2$ and $E_1 \restriction X_2$, there is a 0-definable
  full open subset $X_3 \subseteq X_2$ such that $X_3/E_1$ is locally
  Euclidean.

  Take $X' = X_3$.  The relations ``full subset'' and ``open subset''
  are transitive, so $X_3$ is a full open subset of $X$.  By
  Lemma~\ref{niceness-of-o-eq},
  \begin{itemize}
  \item $E \restriction X_3 = E_1 \restriction X_3$ is an open
    equivalence relation on $X_3$, because $X_3$ is open in $X_1$.
  \item The inclusion $X_3/E \hookrightarrow X_2/E$ is an open
    embedding, and therefore $X_3/E$ is Hausdorff.
  \end{itemize}
\end{proof}

\section{Tameness in the quotient topology}
\label{sec:admiss}
Say that a topology on an interpretable set $Y$ is \emph{admissible}
if it is Hausdorff and there is a definable set $X \subseteq M^n$ and
a surjective definable (continuous) open map $X \twoheadrightarrow Y$
where $X$ has the subspace topology from $X \subseteq M^n$.
Admissible topologies are definable by Lemma~\ref{quot-def}.  The
quotient topologies of Theorem~\ref{better-theorem} are admissible and
locally Euclidean.

\begin{remark}
  If $Y_1$ and $Y_2$ are two interpretable sets with admissible
  topologies, then the disjoint union topological space $Y_1 \coprod
  Y_2$ is also admissible.
\end{remark}
We leave the proof as an exercise to the reader.

\begin{proposition} \label{adm-exist}
  Every interpretable set admits an admissible locally Euclidean
  topology.
\end{proposition}
\begin{proof}
  If $f : X \to Y$ is a definable surjection from a definable set to
  an interpretable set, then $Y$ admits an admissible locally
  Euclidean topology.  We prove this by induction on $\dim(X)$.  By
  Theorem~\ref{better-theorem}, there is an open subset $X' \subseteq
  X$ such that the quotient topology on $f(X')$ is admissible and
  locally Euclidean, and such that $\dim(X \setminus X') < \dim(X)$.
  Let $Y' = f(X')$.  Then
  \begin{equation*}
    f^{-1}(Y \setminus Y') \subseteq X \setminus X'
  \end{equation*}
  Therefore, the inductive hypothesis can be applied to the surjection
  \begin{equation*}
    f^{-1}(Y \setminus Y') \twoheadrightarrow Y \setminus Y',
  \end{equation*}
  showing that $Y \setminus Y'$ admits an admissible locally Euclidean
  topology.  Taking the disjoint union of this topological space with
  the quotient topology on $Y' = f(X')$ gives an admissible topology
  on $Y$.
\end{proof}

We now show that admissible locally Euclidean topologies have some
tameness properties.

\begin{proposition} \label{subspace-adm}
  If $Y$ is an interpretable set with an admissible topology, then the
  subspace topology on any definable subset of $Y$ is also admissible.
\end{proposition}
\begin{proof}
  Let $Y' \subseteq Y$ be a definable subset.  Let $f : X
  \twoheadrightarrow Y$ be the surjection witnessing that the topology
  on $Y$ is admissible.  Let $X' = f^{-1}(Y')$.  Note that $f(U \cap
  X') = f(U) \cap Y'$ for any $U \subseteq X$.  Therefore
  \begin{align*}
    \{U \subseteq Y' : U \text{ is open in } Y'\} & = \{ U \cap Y' : U \text{ is open in } Y\} \\
    &= \{f(U) \cap Y' : U \text{ is open in } X\} \\
    &= \{f(U \cap X') : U \text{ is open in } X\} \\
    &= \{f(U) : U \text{ is open in } X'\}
  \end{align*}
  It follows that the map $f \restriction X'$ is an open map from $X'$
  to $Y'$.  Subspaces of Hausdorff spaces are Hausdorff, so $Y'$ is
  admissible.
\end{proof}

\begin{proposition} \label{finite-pi-0}
  If $Y$ is an interpretable set with an admissible topology, then
  every definable subset of $Y$ can be written as a finite union of
  definably connected sets.
\end{proposition}
\begin{proof}
  By Proposition~\ref{subspace-adm}, it suffices to show that $Y$ itself can
  be written as a finite union of definably connected sets. Let $X \to
  Y$ be a map witnessing admissibility, with $X \subseteq M^n$.  Then
  $X$ has finitely many definably connected components by cell
  decomposition.  The image of a definably connected set under a
  definable continuous map is definably connected, so $Y$ also has
  finitely many definably connected components.
\end{proof}

\begin{lemma} \label{indy2}
  Assume $\aleph_1$-saturation.  Let $Y$ be an interpretable set with
  an admissible topology, as witnessed by some map $f : X \to Y$.  Let
  $S$ be a countable set of parameters over which $f, X, Y$ are defined,
  and $T$ be a countable set.  For any point $p$ and any neighborhood $N$
  of $p$, there is a smaller neighborhood $N' \subseteq N$ of $p$ such that
  \begin{equation*}
    \ulcorner N' \urcorner \forkindep_S pT
  \end{equation*}
\end{lemma}
\begin{proof}
  Let $\tilde{p}$ be some point in $X$ mapping to $p$.  The set
  $f^{-1}(N)$ is an open neighborhood of $\tilde{p}$, because $f$ is
  continuous.  By Lemma~\ref{tricks}(\ref{trick-indy}) there is some
  smaller neighborhood $\tilde{p} \in U \subseteq f^{-1}(N)$ such that $\ulcorner U \urcorner \forkindep \tilde{p}TS$.  Let $N' = f(U)$.
  This is a neighborhood of $p$ because $f$ is an open map.
  Furthermore,
  \begin{equation*}
    \ulcorner U \urcorner \forkindep \tilde{p}TS \implies \ulcorner U \urcorner \forkindep_S \tilde{p}T \implies \ulcorner N' \urcorner \forkindep_S pT
  \end{equation*}
  because $N'$ is defined from $U$ and $p$ is defined from
  $\tilde{p}$.
\end{proof}

If $X$ is an interpretable set with a definable topology, it makes
sense to talk about the ``local dimension'' $\dim_pX$ of $X$ at any
point $p \in X$.  Namely, the local dimension is the minimum of
$\dim(N)$ as $N$ ranges over neighborhoods of $p$ in $X$.  We can also
talk about the local dimension $\dim_Dp$ of a definable subset $D
\subseteq X$ at a point $p \in D$.  Specifically,
\begin{equation*}
  \dim_pD := \min_N \dim(N \cap D) \qquad \qquad \text{ $N$ a neighborhood of $p$ in $X$}.
\end{equation*}
This is the same as the local dimension at $p$ within the subspace
topology on $D$.

\begin{proposition} \label{local-dim}
  Let $Y$ be an interpretable set with an admissible topology.  If $D$
  is any definable subset of $Y$, then
  \begin{equation*}
    \dim(D) = \max_{p \in D} \dim_p(D).
  \end{equation*}
\end{proposition}
\begin{proof}
  By Proposition~\ref{subspace-adm} we may assume $D = Y$.  The
  maximum of the local dimensions is certainly at most $\dim(Y)$, so
  we only need to show that there is some point $p \in Y$ at which
  $\dim_p(Y) = \dim(Y)$.  Because of the definability of dimension,
  \begin{equation*}
    \{p \in Y : \dim_p(Y) = k\}
  \end{equation*}
  is definable for each $k$, in particular for $k = \dim(Y)$.
  Therefore we may pass to an $\aleph_1$-saturated elementary
  extension.  Let $S$ be a finite set of parameters over which $Y$ is
  defined, and let $p \in Y$ be a point such that $\dim(p/S) =
  \dim(Y)$.  We claim that the local dimension of $Y$ at $p$ is
  $\dim(Y)$.  Let $N \subseteq Y$ be any neighborhood of $p$; we will
  show that $\dim(N) = \dim(Y)$.  By Lemma~\ref{indy2}, there is a
  smaller neighborhood $N'$ of $p$ such that
  \begin{equation*}
    \ulcorner N' \urcorner \forkindep_S p
  \end{equation*}
  Therefore, $\dim(p/\ulcorner N' \urcorner S) = \dim(p/S) = \dim(Y)$.
  Because $p$ lies in $N'$,
  \begin{equation*}
    \dim(N') \ge \dim( p/\ulcorner N' \urcorner S) = \dim(p/S) = \dim(Y).
  \end{equation*}
  On the other hand, $Y \supseteq N \supseteq N'$, so
  \begin{equation*}
    \dim(Y) \ge \dim(N) \ge \dim(N').
  \end{equation*}
  Therefore the inequalities are equalities and $\dim(N) = \dim(Y)$.
  As $N$ was an arbitrarily small neighborhood of $p$, it follows that
  the local dimension $\dim_p(Y)$ agrees with $\dim(Y)$.
\end{proof}

\begin{proposition} \label{dim-bound}
  Let $Y$ be an interpretable set with an admissible locally Euclidean
  topology.
  \begin{enumerate}
  \item If $D$ is any definable subset of $Y$, then $\dim \partial D <
    \dim D$.
  \item If $D$ is any definable subset of $Y$, then $\dim \overline{D}
    = \dim D$ and $\dim \bd(D) < \dim Y$.
  \item \label{tres} Assuming saturation: if $D$ is a type-definable
    subset of $Y$ of dimension $d$, then $D$ is contained in a
    definable closed set of dimension $d$.
  \end{enumerate}
\end{proposition}
\begin{proof}~
  \begin{enumerate}
  \item Let $k = \dim \partial D$.  By Proposition~\ref{local-dim},
    there is a point $x \in \partial D$ such that $\dim_x(\partial D)
    = k$.  Let $U$ be an open neighborhood of $x$ which is definably
    homeomorphic to an open subset of $M^n$ for some $n$.
    Transferring the situation along the homeomorphism, and using the
    analogous fact for definable sets (=
    Lemma~\ref{tricks}(\ref{trick-dim}) or \cite{lou-o-minimality}
    Theorem 4.1.8), we see that $\dim(D \cap U) > k$.
  \item These bounds follow because $\overline{D} = D \cup \partial
    D$, and $\bd(D) = \partial D \cup \partial (Y \setminus
    D)$.
  \item By general properties of dimension, $D \subseteq D'$ for some
    definable subset $D'$ of dimension $d$.  Then $D \subseteq
    \overline{D'}$ and $\dim \overline{D'} = d$.
  \end{enumerate}
\end{proof}

Using Lemma~\ref{indy2} and Proposition~\ref{dim-bound}, one can
transfer facts about ``generic behavior'' from the definable setting
to the admissible interpretable setting.  We give two examples:
\begin{itemize}
\item Definable subsets are Euclidean at generic points
\item Definable functions are continuous at generic points in their
  domain.
\end{itemize}

\begin{remark}\label{euclid-local-euclid}
  If $D$ is any definable subset of $M^n$, then there is a definable
  full open subset $D' \subseteq D$ such that $D'$ is locally
  Euclidean as a subspace of $M^n$.  (Here, we mean that $D'$ is open
  in $D$, not open in $M^n$.)
\end{remark}
\begin{proof}
  Write $D$ as a disjoint union of cells $\bigcup_{i = 1}^k C_i$ by
  cell decomposition.  Each cell $C_i$ is locally Euclidean in
  isolation.  Take $D'$ to be $D \setminus \overline{\bigcup_{i = 1}^k
    \partial C_i}$, where the closure and frontier are with respect to the topology on $M^n$.
		This is open as a subset of $D'$, and a full open
  subset by standard dimension bounds.  Every point $p$ in $D'$ is in
  the closure of exactly one $C_i$, so the Euclideanity of $C_i$ at
  $p$ implies the Euclideanity of $D'$ at $p$.
\end{proof}

\begin{lemma} \label{thorn}
  Let $Y$ be an interpretable set with admissible locally Euclidean
  topology.  If $D$ is any definable subset of $Y$, then there is a
  definable full open subset $D' \subseteq D$ such that the subspace
  topology on $D'$ is locally Euclidean.
\end{lemma}
\begin{proof}
  Without loss of generality, we may assume the ambient model is
  sufficiently saturated, and that $Y$ and $D$ are 0-definable.
  \begin{claim}
    If $a \in D$ is generic in $D$, then $D$ is locally Euclidean at
    $a$.
  \end{claim}
  \begin{proof}
    By local Euclideanity of $Y$, there is an open neighborhood $a \in
    U \subseteq Y$ definably homeomorphic to an open in some $M^n$.
    By Lemma~\ref{indy2}, we may shrink $U$ and assume that $a
    \forkindep \ulcorner U \urcorner$.  Let $\iota : U \hookrightarrow
    M^n$ be the definable open embedding.  Moving $\iota$ by an
    automorphism fixing $\ulcorner U \urcorner$, we may assume $a
    \forkindep \ulcorner U \urcorner \ulcorner \iota \urcorner$.

    Therefore we can name $\ulcorner U \urcorner$ and $\ulcorner \iota
    \urcorner$ as constants, and assume that $U$ and $\iota$ are
    0-definable, without losing the fact that $a$ is generic.  Now
    because $a$ is generic in $U \cap D$, the image $\iota(a)$ is
    generic in $\iota(U \cap D)$.  By
    Remark~\ref{euclid-local-euclid}, $\iota(U \cap D)$ is Euclidean
    at $\iota(a)$.  Transferring things back along $\iota^{-1}$, we
    see that $D$ is Euclidean at $a$, proving the claim.
  \end{proof}
  Now let $D''$ be the locally Euclidean locus of $D$.  Then $D''$ is
  an ind-definable subset of $D$, i.e., $D \setminus D''$ is a
  type-definable set.  By the claim, $D \setminus D''$ has lower
  dimension than $D$.  Thus by Proposition \ref{dim-bound}(\ref{tres})
  there is a definable closed set $F$ containing $D \setminus D''$,
  with $\dim F < \dim D$.  Take $D' = D \setminus F$.  Then $D'$ is a
  full open subset of $D$, and $D' \subseteq D''$.
\end{proof}

We recall another basic fact about o-minimality:
\begin{remark} \label{basic-fact}
  Let $U$ and $U'$ be 0-definable open subsets of powers of $M$, and
  let $f$ be a 0-definable partial map from $U$ to $U'$.  Suppose $a$
  is generic in $U$ and that $f$ is defined at $a$.  Then $f$ is
  defined and continuous on an open neighborhood of $a$.
\end{remark}

\begin{proposition} \label{oof}
  Let $Y$ and $Y'$ be two interpretable sets with admissible locally
  Euclidean topologies, and $f$ be a definable map from $Y$ to $Y'$.
  Then $Y$ can be written as a finite disjoint union of definable
  locally closed sets, on which the restriction of $f$ is continuous.
  More generally, this holds when $Y$ is a definable subspace of an
  admissible locally Euclidean space.
\end{proposition}
\begin{proof}
  By induction on $\dim(Y)$ it suffices to show that $f$ is continuous
  on a full open subset of $Y$.  By Lemma~\ref{thorn}, we may assume
  $Y$ is locally Euclidean.  By Proposition~\ref{dim-bound}, the
  interior of any full subset is a full subset of $Y$, so it suffices
  to show that the continuous locus of $f$ is a full subset of $Y$.

  Without loss of generality, we may assume that everything is defined
  over $\emptyset$ and that the ambient model is $\aleph_1$-saturated.
  It suffices to show that $f$ is continuous at generic points of $Y$.
  Fix some generic $y \in Y$.  By locally Euclideanity, there are open
  neighborhoods $U$ of $y$ and $U'$ of $f(y)$ admitting open
  embeddings into powers of $M$.  By Lemma~\ref{indy2} we may shrink
  $U$ and $U'$ in such a way that
  \begin{equation*}
    \ulcorner U \urcorner \forkindep y \text{ and }
    \ulcorner U' \urcorner \forkindep y \ulcorner U \urcorner
  \end{equation*}
  Thus $\ulcorner U \urcorner \ulcorner U' \urcorner \forkindep y$.
  Let $\iota$ and $\iota'$ be definable open embeddings from $U$ and
  $U'$ into powers of $M$.  Moving $\iota$ and $\iota'$ by an
  automorphism over $\ulcorner U \urcorner \ulcorner U' \urcorner$, we
  may assume that
  \begin{equation*}
    \ulcorner \iota \urcorner \ulcorner \iota' \urcorner
    \forkindep_{\ulcorner U \urcorner \ulcorner U' \urcorner} y
  \end{equation*}
  and so
  \begin{equation*}
    \ulcorner \iota \urcorner \ulcorner \iota' \urcorner \ulcorner U
    \urcorner \ulcorner U' \urcorner \forkindep y
  \end{equation*}
  By naming constants, we may assume that $U$, $U'$, $\iota$, $\iota'$
  are all 0-definable, and $y$ is still generic in $Y$.

  Now $f \restriction (U \cap f^{-1}(U'))$ is a partial function from
  $U$ to $U'$, defined at $y$.  Transferring things along the open
  embeddings $\iota$, $\iota'$, we reduce to the case where $U$ and
  $U'$ are open subsets of powers of $M$, reducing to the situation of
  Remark~\ref{basic-fact} above.
\end{proof}

This completes the proof of Theorem~\ref{addendum} in the
introduction, which is merely a compilation of Proposition
\ref{adm-exist}, Lemma~\ref{quot-def}, Propositions \ref{finite-pi-0},
\ref{dim-bound}, \ref{local-dim}, and \ref{oof}.

We close with a few open questions:
\begin{question}
Do Propositions~\ref{dim-bound} and \ref{oof} hold without the local Euclideanity assumption?
\end{question}
\begin{question}
For ``definable spaces'' in the sense of Peterzil and Steinhorn \cite{pz-definable-space}, is our definition of ``definable compactness'' (Definition~\ref{d3.2}) equivalent to Peterzil and Steinhorn's definition using completable curves?
\end{question}

\bibliographystyle{plain}
\bibliography{mybib}{}

\end{document}